\documentclass[fleqn,preprint,12pt]{elsarticle}
\usepackage{latexsym,amsmath,amssymb,epsfig,amsthm}
\usepackage{comment}

\newtheorem{lemma}{Lemma}

\theoremstyle{definition}
\newtheorem{remark}{Remark}

\usepackage{comment}
\usepackage{tabularx}
\usepackage{algorithm}
\usepackage{algorithmic}

\newcolumntype{Y}{>{\small\center\arraybackslash}X}

\algsetup{indent=2em}

\newcommand\nc{\newcommand}

\nc\ynm{Y_{nm}}
\nc\bynm{{\mathbf{Y}}_{nm}}
\nc\psinm{{\mathbf{\Psi}}_{nm}}
\nc\phinm{{\mathbf{\Phi}}_{nm}}

\nc\fnm{f_{nm}}
\nc\gnm{g_{nm}}
\nc\unm{u_{nm}}
\nc\vnm{v_{nm}}

\nc\hfnm{\hat{f}_{nm}}
\nc\hgnm{\hat{g}_{nm}}
\nc\hunm{\hat{u}_{nm}}
\nc\hvnm{\hat{v}_{nm}}

\nc\anj{a_{n,j}}
\nc\bnj{b_{n,j}}
\nc\alnj{\alpha_{n,j}}
\nc\alnk{\alpha_{n,k}}
\nc\benj{\beta_{n,j}}
\nc\benk{\beta_{n,k}}

\nc\ba{{\mathbf{A}}}
\nc\be{{\mathbf{E}}}
\nc\bh{{\mathbf{H}}}

\nc\bn{{\mathbf{n}}}

\nc\br{{\mathbf{r}}}
\nc\rhat{\hat{\br}}
\nc\curl{\nabla\times}
\nc\dive{\nabla\cdot}
\nc\pt{\frac{\partial}{\partial t}}
\nc\pr{\frac{\partial}{\partial r}}

\nc\sumnm{\sum_{n=0}^\infty\sum_{m=-n}^n}
\nc\sumon{\sum_{n=0}^\infty}
\nc\sumom{\sum_{m=-n}^n} 

\nc\prodj{\prod_{j=1}^n}
\nc\prok{\prod_{k=1}^n}
\nc\prokj{\prod_{k=1,k\neq j}^n}
\nc\sumj{\sum_{j=1}^n}

\nc\cnhat{\hat{C}_n}

\begin{document}

\begin{frontmatter}

\title{The solution of the scalar wave equation in the 
exterior of a sphere} 
\author[lg]{Leslie Greengard\fnref{fnlg}}
\address[lg]{Courant Institute of Mathematical Sciences, New York University,
New York, NY 10012.}
\ead{greengard@courant.nyu.edu}
\fntext[fnlg]{Supported in part by 
the NSSEFF Program under AFOSR Award FA9550-10-1-0180 and
in part by the Department of Energy under contract DEFG0288ER25053.}
\author[th]{Thomas Hagstrom\fnref{fnth}}
\address[th]{Department of Mathematics, Southern Methodist University,
PO Box 750156, Dallas, TX 75275.}
\ead{thagstrom@smu.edu}
\fntext[fnth]{Supported in part by the Army Research Office under 
agreement W911NF-09-1-0344 and by the NSF under grant OCI-0904773.}
\author[sj]{Shidong Jiang\corref{cor1}\fnref{fn1}}
\address[sj]{Department of Mathematical Sciences, New Jersey Institute of
Technology, Newark, New Jersey 07102.}
\ead{shidong.jiang@njit.edu}
\cortext[cor1]{Corresponding author.}
\fntext[fn1]{Supported in part by NSF under grant CCF-0905395.}

\begin{abstract}
We derive new, explicit representations for the  
solution to the scalar wave equation in the exterior
of a sphere, subject to either Dirichlet or Robin boundary conditions.
Our formula leads to a stable and high-order 
numerical scheme that permits the evaluation of
the solution at an arbitrary target, without the use of a spatial
grid and without numerical dispersion error. 
In the process, we correct some errors in the analytic literature
concerning the asymptotic behavior of the logarithmic derivative
of the spherical modified Hankel function. 
We illustrate the performance of the method with several numerical
examples.
\end{abstract}


\end{frontmatter}

\section{Introduction}

In this paper, we consider a simple problem, namely the solution
of the scalar wave equation
\begin{equation}\label{eq2.1}
u_{tt}=\Delta u, \quad t>0,
\end{equation}
subject to homogeneous initial conditions
\begin{equation}\label{eq2.2}
u(r,\theta,\phi,0)=0, \qquad u_t(r,\theta,\phi,0)=0 
\end{equation}
in the exterior of the unit sphere. Here, $(r,\theta,\phi)$ denote the
spherical coordinates of a point in ${\mathbb R}^3$ with $r>1$. 
Standard textbooks 
on mathematical physics (such as \cite{CH,MF}) 
present exact solutions
for the time-harmonic cases governed by the 
Helmholtz equation, but generally fail to discuss 
the difficulties associated with the 
fully time-dependent case \eqref{eq2.1}.
As we shall see, it is a nontrivial matter to develop 
closed-form solutions, and a surprisingly subtle matter to develop solutions 
that can be computed without catastrophic cancellation.

In this paper, we restrict our attention to 
boundary value problems with Dirichlet or Robin conditions.
We consider the Dirichlet problem first, and assume we are given 
data on the boundary of the unit sphere of the form:
\begin{equation}\label{eq2.3}
u(1,\theta,\phi,t)=f(\theta,\phi,t).
\end{equation}
It is natural to begin by expanding both 
$u$ and $f$ in terms of spherical harmonics. 
\begin{equation}\label{eq2.4}
\begin{aligned}
u(r,\theta,\phi,t)&=\sumnm \unm(r,t) \ynm(\theta,\phi),\\
f(\theta,\phi,t)&=\sumnm \fnm(t) \ynm(\theta,\phi),
\end{aligned}
\end{equation}
where 
\begin{equation}
Y_n^m(\theta,\phi) =
\sqrt{\frac{2n+1}{4 \pi}}
\sqrt{\frac{(n-|m|)!}{(n+|m|)!}}\, P_n^{|m|}(\cos \theta)
		      e^{i m \phi} \, ,
\label{ynmpnm}
\end{equation}
$P_n(x)$ is the standard Legendre polynomial of degree $n$,
and the associated Legendre functions $P_n^m$ are
defined by the Rodrigues' formula
\[ P_n^m(x) = (-1)^m (1-x^2)^{m/2} \frac{d^m}{dx^m} P_n(x). \] 

We let $\hunm(r,s)$ and $\hfnm(s)$ denote the Laplace transforms of 
$\unm(r,t)$ and $\fnm(t)$:
\begin{equation}\label{eq2.5}
\hunm(r,s)=\int_0^\infty e^{-st}\unm(r,t)dt,
\end{equation}
\begin{equation}\label{eq2.6}
\hfnm(s)=\int_0^\infty e^{-st}\fnm(t)dt.
\end{equation}
It is straightforward \cite{handbook} to see that $\hunm(r,s)$ satisfies the
linear second order ordinary differential equation (ODE)
\[ r^2 \hunm(r,s)_{rr} + 2r \hunm(r,s)_r - [s^2r^2 + n(n+1)] \hunm(r,s) = 0,
\]
for which the decaying solution as $r \rightarrow \infty$ is 
the modified spherical Hankel function $k_n(sr)$.
It follows that 
\[ \hunm(r,s)=c_{nm}(s)k_n(sr).  \]
Matching boundary data on the unit sphere, 
we have $c_{nm}(s)=\hfnm(s)/k_n(s)$, and 
\begin{equation}\label{eq2.7}
\hunm(r,s)=\frac{k_n(sr)}{k_n(s)}\hfnm(s).
\end{equation}

The remaining difficulty is that we have an explicit solution in the 
Laplace transform domain, but we seek the solution in the time domain.
For this, we write
the right hand side of \eqref{eq2.7} in a form for which
the inverse Laplace transform can carried out analytically.
First, from \cite{handbook,jiang1,olver}, we have
\begin{equation}\label{eq2.8}
k_n(z)=\frac{p_n(z)}{z^{n+1}}e^{-z}=\frac{\prodj (z-\alnj)}{z^{n+1}}e^{-z},
\end{equation}
where $\alnj$ ($j=1,\cdots,n$) are the simple roots of $k_n$ 
lying on the open left half of the complex plane.
Thus,
\begin{equation}\label{eq2.9}
\begin{aligned}
\frac{k_n(sr)}{k_n(s)}&=\frac{1}{r}e^{-s(r-1)}
\prodj \frac{s-\frac{1}{r}\alnj}{s-\alnj}\\
&=\frac{1}{r}e^{-s(r-1)}\left(1+\sumj \frac{\anj(r)}{s-\alnj}\right),
\end{aligned}
\end{equation}
where the second equality follows from an expansion using partial fractions 
and the coefficients $\anj$ are given from the residue theorem by the formula: 
\begin{equation}\label{eq2.10}
\begin{aligned}
\anj(r)&=\frac{\prok(\alnj-\frac{1}{r}\alnk)}{\prokj (\alnj-\alnk)}\\
&=\frac{p_n(\alnj r)}{r^n p'_n(\alnj)}\\
&=re^{\alnj (r-1)}\frac{k_n(\alnj r)}{k'_n(\alnj)}, \qquad j=1,\cdots,n.
\end{aligned}
\end{equation}
Substituting \eqref{eq2.9} into \eqref{eq2.7}, we obtain
\begin{equation}\label{eq2.11}
\hunm(r,s)=\frac{1}{r}\left(1+\sumj \frac{\anj(r)}{s-\alnj}\right)
(e^{-s(r-1)}\hfnm(s)).
\end{equation}
Taking the inverse Laplace transform of both sides, we have
\begin{equation}\label{eq2.12}
\unm(r,t)=\frac{1}{r}\left(\fnm(t-r+1)+\sumj 
\anj(r) \int_0^{t-r+1}
e^{\alnj (t-r+1-\tau)}\fnm(\tau)d\tau\right).
\end{equation}
This involves the use of the convolution theorem and the formulas
$\mathcal{L}^{-1}\left(\frac{1}{s-\alpha}\right)=e^{\alpha t}$ and
$\mathcal{L}^{-1}\left(e^{-s(r-1)}\hfnm(s)\right)=\fnm(t-r+1)H(t-r+1)$, 
where $H$ is the Heaviside function.

\begin{remark}
Wilcox \cite{wilcox} studied the solution of the scalar wave equation
and derived formula \eqref{eq2.12} in 1959.
In that short note, Wilcox stated that the coefficients 
$\anj$ given by \eqref{eq2.10} grew slowly based on the claim that 
$\frac{k_n(\alnj)}{k_n'(\alnj)}=O(n^{1/2})$ as $n\rightarrow \infty$. 
Unfortunately, this estimate is incorrect. In fact,
even after multiplication by the exponentially decaying factor $e^{\alnj(r-1)}$,
the coefficients $\anj$ ($j=1,\cdots,n$) grow exponentially fast
as $n\rightarrow \infty$. In the next section, we explain this growth 
in detail. As a result, even though
\eqref{eq2.12} is very convenient for the purpose of theoretical studies, 
it cannot be used for numerical calculation due to 
catastrophe cancellation in carrying out the summation. 
\end{remark}

\begin{remark}
Benedict, Field and Lau \cite{BFL-QG} have recently developed algorithms for compressing 
the kernel, which they call the teleportation kernel, arising in sphere-to-sphere propagation
of data both for the standard wave equation as well as wave equations arising in linearized
gravitational theories. For the wave equation their compressed kernels can be
used to perform the same function as our
solution of the Dirichlet problem. The largest value of $n$ considered in \cite{BFL-QG}   
is $64$. It is as yet unclear if useful compressions for much larger values of $n$ can be
constructed using their methods.
\end{remark} 
\subsection{Asymptotic growth of the logarithmic derivative
of the spherical modified Bessel function}

We first show that the coefficients $\anj$ ($j=1,\cdots,n$) defined in 
\eqref{eq2.10} grow exponentially 
as $n\rightarrow \infty$, for fixed large $r$.
Indeed, Lemma \ref{lem6.1} in Section \ref{sec:num} shows that the zeros $\alnj$ of
$k_n$ satisfy the estimates: 
$|\alnj| \sim O(n)$ for all $j$ and 
$|\alnj-\alnk| \propto |j-k|$. Thus when $r$ is large, we have
\begin{equation}\label{eq2.21}
\begin{aligned}
\max_j |\anj(r)| &= \max_j \left|\frac{\prok(\alnj-\frac{1}{r}\alnk)}
{\prokj (\alnj-\alnk)}\right|\\
&\sim \frac{n^n}{n!}\\
&\sim e^n,
\end{aligned}
\end{equation}
where the last line follows from Stirling's formula
$n!\sim \sqrt{2\pi n}\left(\frac{n}{e}\right)^n$.
We have computed $\max_j|\anj|$ for $n=1,\cdots,200$ using 
\eqref{eq2.10}, and plotted them
in Figure \ref{fig2.1} for $r=2$, clearly exhibiting
the exponential growth of $\max_j |\anj|$.
We also plot $|\anj(r)|$ as a function of $j$ for a fixed value of $n$
in Fig. \ref{fig2.2}.

\begin{figure}[!ht]
\centerline{\includegraphics[trim=120 240 120 220, clip, height=2.5in]
{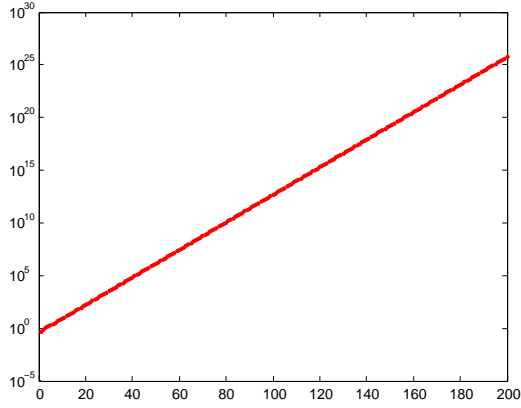}} 
\caption{The function \emph{$\max_j|\anj(r)|$} for increasing values of
$n$, with $r=2$. 
} \label{fig2.1}
\end{figure}

\begin{figure}[!ht]
\centerline{\includegraphics[trim=120 240 120 220, clip, height=2.5in]
{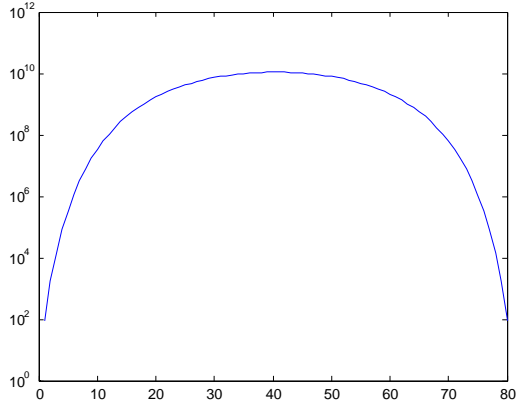}} 
\caption{A plot of $|\anj(r)|$ as a function of $j$, 
for $n=80$ and $r=2$.
} \label{fig2.2}
\end{figure}

From the preceding analysis, it is clear that one cannot use \eqref{eq2.12}
as stated, since the desired solution 
is $O(1)$ and catastrophic cancellation
will occur in computing $u(r,\theta,\phi,t)$ 
from exponentially large intermediate quantities.

Fortunately, even though $\max_j|\anj(r)|$ 
grows exponentially as $n$ increases, 
we can rewrite \eqref{eq2.12} in the form of a convolution, which involves much more 
benign growth: 
\begin{equation}\label{eq2.12.1}
\unm(r,t)=\frac{1}{r}\left(\fnm(t-r+1)+
\int_0^{t-r+1}C_n(r,t-r+1-\tau)\fnm(\tau)d\tau\right),
\end{equation}
where the convolution kernel $C_n$ is defined by the formula
\begin{equation}\label{eq2.14}
C_n(r,t)=\sumj \anj(r)e^{\alnj t}.
\end{equation}
If we write
\begin{equation}\label{eq2.15}
C_n(r,t)=\mathcal{L}^{-1}\left(\cnhat(r,s)\right),
\end{equation}
then from \eqref{eq2.10}, we have
\begin{equation}\label{eq2.16}
\begin{aligned}
\cnhat(r,s)&=\prodj \frac{s-\frac{1}{r}\alnj}{s-\alnj}-1
=\sumj \frac{\anj(r)}{s-\alnj}\\
&=re^{s(r-1)}\frac{k_n(sr)}{k_n(s)}-1\\
&=\sqrt{r}e^{s(r-1)}\frac{K_{n+1/2}(sr)}{K_{n+1/2}(s)}-1,
\end{aligned}
\end{equation}
where $K_{n+1/2}$ 
is the modified Bessel function of the second kind. The last expression follows from
the fact that $k_n(z)=\sqrt{\frac{2}{\pi z}}K_{n+\frac{1}{2}}(z)$.

The convolution kernel $C_n(r,t)$ and its Laplace transform $\cnhat(r,s)$ 
are plotted in Figs. \ref{fig2.3} and \ref{fig2.4}, respectively. 
\begin{figure}[!ht]
\centerline{\includegraphics[trim=120 240 120 220, clip, height=2.5in]
{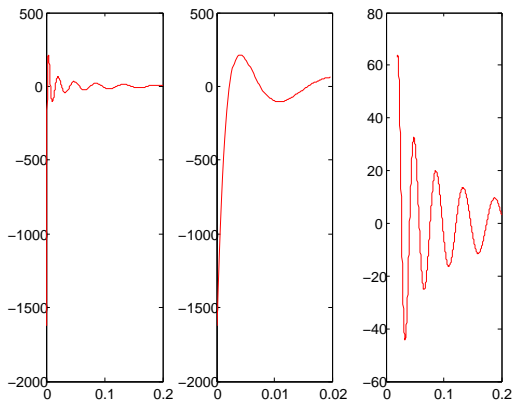}} 
\caption{\emph{The convolution kernel $C_n(r,t)$ as a function of $t$ for
$n=80$ and $r=2$. The left-hand plot shows $C_n(r,t)$ for $t\in [0,0.2]$,
the middle plot shows the same function on $[0,0.02]$, and the 
right-hand plot shows the function on $[0.02,0.2]$.}} 
\label{fig2.3}
\end{figure}
\begin{figure}[!ht]
\centerline{\includegraphics[trim=120 240 120 220, clip, height=2.5in]
{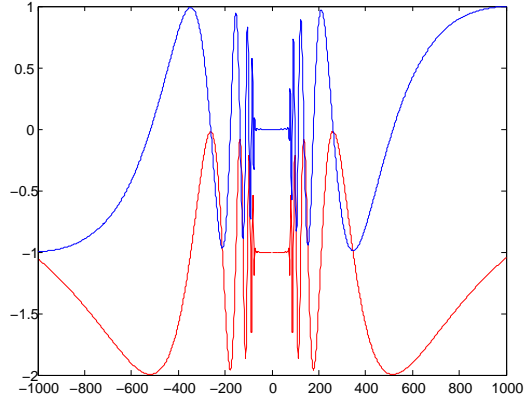}} 
\caption{The Laplace transform $\cnhat(r,s)$ of the convolution kernel 
is plotted on the imaginary axis over the range $[-1000i,1000i]$ 
for $n=80$ and $r=2$. The red (lower) curve corresponds to the real part
of $\cnhat(r,s)$ and
the blue (upper) curve corresponds to its imaginary part.} 
\label{fig2.4}
\end{figure}

The following lemma shows that the convolution kernel grows only quadratically
as a function of $n$ at $t=0$. 
Numerical experiments
(see Fig. \ref{fig2.3}) suggest that $C_n(r,t)$ is maximal in magnitude at $t=0$.
Thus, while the sum of exponential expression (\ref{eq2.14}) involves catastrophic
cancellation, the function $C_n(r,t)$ is, itself, well-behaved and we may seek
an alternative method for the evaluation of the convolution integral.

\begin{lemma}
Let $r>1$. Then
\begin{equation}\label{eq2.17}
C_n(r,0)=\sumj \anj(r)=\frac{n(n+1)}{2}\left(\frac{1}{r}-1\right).
\end{equation}
\end{lemma}
\begin{proof}
By the initial value theorem for the Laplace transform, 
\begin{equation}\label{eq2.18}
C_n(r,0)=\lim_{s\rightarrow \infty} s\cnhat(r,s).
\end{equation}
The first equality in \eqref{eq2.17} follows from 
\eqref{eq2.16}.
From \cite{handbook} (formula 9.7.2 on page 378), we have 
the asymptotic expansion 
\begin{equation}\label{eq2.19}
K_\nu(z)\sim \sqrt{\frac{\pi}{2z}}e^{-z}
\left\{1+\frac{\mu-1}{8z}+\frac{(\mu-1)(\mu-9)}{2!(8z)^2}+\cdots\right\},
\end{equation} 
where $\mu=4\nu^2$.
Substituting \eqref{eq2.16} and \eqref{eq2.19} into \eqref{eq2.18}, 
we obtain
\begin{equation}\label{eq2.20}
C_n(r,0)=\lim_{s\rightarrow \infty} s\left(\frac{1+\frac{\mu-1}{8sr}+O(s^{-2})}
{1+\frac{\mu-1}{8s}+O(s^{-2})}-1
\right)=\frac{\mu-1}{8}\left(\frac{1}{r}-1\right).
\end{equation}
The result \eqref{eq2.17} now follows from the fact that 
$\mu=4\nu^2=4\left(n+\frac{1}{2}\right)^2$.
\end{proof}

Despite the fact that $\max_j|\anj(r)|$ grows exponentially with $n$,
\eqref{eq2.17} shows that the {\em sum of weights} $\anj$ is only $O(n^2)$ for 
fixed $r$. Still, however, the formula 
\eqref{eq2.12} cannot be used in practice because of catastrophic 
cancellation in the summation 
\[ \sumj \anj(r) \int_0^{t-r+1} e^{\alnj (t-r+1-\tau)}\fnm(\tau)d\tau. \] 
Thus, we will need a different representation for the convolution 
operator $\int_0^{t-r+1} C_n(r,t-r+1-\tau)\fnm(\tau)d\tau$ 
which is suitable for numerical computation.

\subsection{Stable computation of the convolution integral} \label{sec:recurrence}

To obtain a stable formula, we note first that we may
rewrite \eqref{eq2.12} in the form:
\begin{equation}\label{eq3.2.1}
\unm(r,t)=\frac{1}{r}
\int_0^{t-r+1}\mathcal{L}^{-1}\left(\cnhat(r,s)+1\right)(r,t-r+1-\tau)
\fnm(\tau)d\tau.
\end{equation}
We then use \eqref{eq2.16} to express 
$\cnhat$ as
\begin{equation}\label{eq3.2.2}
\begin{aligned}
\cnhat(r,s)+1&=\prodj \frac{s-\frac{1}{r}\alnj}{s-\alnj}
&=\prodj \left(1+\frac{(1-\frac{1}{r})\alnj}
{s-\alnj}\right).
\end{aligned}
\end{equation}

We can, therefore, compute $\unm$ recursively:
\begin{equation}\label{eq3.2.3}
\begin{aligned}
\phi_0(t)&=\fnm(t),\\
\phi_j(t)&=\phi_{j-1}(t)+\left(1-\frac{1}{r}\right)\alnj\int_0^t 
e^{\alnj(t-\tau)}\phi_{j-1}(\tau)d\tau, \quad j=1,\cdots,n
\end{aligned}
\end{equation}
and, finally,
\begin{equation}\label{eq3.2.4}
\unm(r,t)=\frac{1}{r}\phi_n(t-r+1).
\end{equation}

Numerical experiments indicate that the above recursion 
is stable if the zeros $\alnj$ of $k_n$ are arranged in ascending
order according to their real parts, i.e., $\alpha_{n,1}$ is closest
to the negative real axis and $\alpha_{n,n}$ is closest to the imaginary
axis. 

\begin{remark}
Alternatively, it is easy to show that the functions $\phi_j$ ($j=1,\cdots,n$) are the solutions 
to the following first order system
of ordinary differential equations (ODEs) with zero initial conditions.
\begin{equation}\label{eq3.2.5}
A\frac{d\phi}{dt} = B\phi + F(t),
\end{equation}
where $\phi$ is a column vector of length $n$ with the $j$th entry being $\phi_j$, $A$, $B$ are $n\times n$ constant matrices defined by the formulas
\begin{equation}\label{eq3.2.6}
A = \left( \begin{array}{cccc}
1 &  & & 0\\
-1 & \ddots &  & \\
& \ddots & \ddots &  \\
0 & & -1 & 1 \end{array} \right),
\qquad
B = \left( \begin{array}{cccc}
\alpha_{n,1} &  & & 0\\
-\frac{\alpha_{n,2}}{r} 
& \ddots &  & \\
& \ddots & \ddots &  \\
0 & & -\frac{\alpha_{n,n}}{r}  & \alpha_{n,n} \end{array} \right),
\end{equation}
and $F$ is a column vector of length $n$ whose only nonzero
entry is $F_1(t)=f'_{nm}(t)-\frac{\alpha_{n,1}}{r}f_{nm}(t)$.
\end{remark}

\begin{remark}
The ODE system \eqref{eq3.2.5} can actually be solved analytically.
That is, one may multiply both sides of \eqref{eq3.2.5} by $A^{-1}$
to obtain 
\begin{equation}\label{eq3.2.7}
\frac{d\phi}{dt}=M\phi+A^{-1}F(t),
\end{equation}
where $M=A^{-1}B$. It is clear that $M$ is a constant lower triangular
matrix. One could then diagonalize the system using the eigen-decomposition 
$M=S\Lambda S^{-1}$. This, however, is numerically unstable since $M$ is a 
highly {\it nonnormal} matrix. Thus, even though
the condition number of $M$ is not very high (numerical evidence shows that 
$\text{cond}(M) = O(n)$), $S$ is extremely ill-conditioned. In fact, more detailed
analysis shows that this approach leads exactly to the formula 
\eqref{eq2.12}. Nevertherless, the ODE system \eqref{eq3.2.5} itself can be solved 
numerically using standard ODE packages,
albeit less efficiently than the explicit recursive approach 
we present in section \ref{sec:num}, especially for high precision.
\end{remark} 
\section{The Robin problem}

In this section, we consider the Robin problem for the scalar
wave equation on the unit sphere:
\begin{equation}\label{eq3.1}
v_{tt}-\Delta v=0, \qquad r>1,\quad t>0,
\end{equation}
with homogeneous initial data
\begin{equation}\label{eq3.2}
v(r,\theta,\phi,0)=0, \qquad v_t(r,\theta,\phi,0)=0, \qquad r>1,
\end{equation}
and the boundary condition
\begin{equation}\label{eq3.3}
\left(\pr+1\right)v(r,\theta,\phi,t)=g(\theta,\phi,t), \qquad r=1.
\end{equation}

It should be noted that Tokita \cite{tokita} extended Wilcox's analysis
of the Dirichlet problem to the case of 
Robin boundary conditions of the form 
$(\pr+\sigma)v=g$, 
although he assumed that $\sigma<1$ in his discussion.
We are primarily concerned with the case $\sigma=1$ since it arises in the 
solution of the full Maxwell equations \cite{GHJ2}.

As in the analysis of the Dirichlet problem, we 
first expand $v$ and $g$ in terms of spherical harmonics, perform
the Laplace transform in $t$, match the boundary data and obtain

\begin{equation}\label{eq3.4}
\begin{aligned}
v(r,\theta,\phi,t)&=\sumnm \mathcal{L}^{-1}(\hvnm(r,s)) \ynm(\theta,\phi),\\
g(\theta,\phi,t)&=\sumnm \mathcal{L}^{-1}(\hgnm(s)) \ynm(\theta,\phi),
\end{aligned}
\end{equation}
and
\begin{equation}\label{eq3.5}
\hvnm(r,s)=\frac{k_n(sr)}{sk'_n(s)+k_n(s)}\hgnm(s).
\end{equation}

We turn now to a study the properties of the kernel in \eqref{eq3.5},
letting
\begin{equation}\label{eq3.8}
\mathcal{K}_n(r,s)=\frac{k_n(sr)}{sk'_n(s)+k_n(s)},
\end{equation}
and
\begin{equation}\label{eq3.9}
D_n(z)=zk'_n(z)+k_n(z).
\end{equation}
Recalling from \ref{eq2.8} that
$k_n(z)=\frac{p_n(z)}{z^{n+1}}e^{-z}$, we have
\begin{equation}\label{eq3.10}
D_n(z) =\frac{-zp_n(z)-np_n(z)+zp'_n(z)}{z^{n+1}}e^{-z}
\equiv \frac{q_{n+1}(z)}{z^{n+1}}e^{-z}.
\end{equation}
Hence,
\begin{equation}\label{eq3.11}
\mathcal{K}_n(r,s)=\frac{1}{r}\frac{p_n(sr)}{q_{n+1}(s)r^n}e^{-s(r-1)}.
\end{equation}
In particular, for $n=0$, we have
\begin{equation}\label{eq3.12}
\mathcal{K}_0(r,s)=-\frac{1}{rs}e^{-s(r-1)}.
\end{equation}
Obviously, the poles of $\mathcal{K}_n$ are simply the zeros of $D_n$.
Those zeros have been characterized by Tokita \cite{tokita} in the following lemma.
\begin{lemma}\label{lem3.1}
\emph{[adapted from \cite{tokita}.]}
For $n\geq 1$, 
$D_n(z)=zk'_n(z)+k_n(z)$ has $n+1$ simple roots denoted 
by $\{\beta_{n,0},\cdots,\beta_{n,n}\}$. All the roots lie in the open
left half of the complex plane symmetrically with respect to the real axis.
Furthermore, they satisfy the following estimates
\begin{equation}\label{eq3.13}
\Re \benj < -A n^{\frac{1}{3}},
\end{equation}
\begin{equation}\label{eq3.14}
|\benj| < B n,
\end{equation}
for sufficiently large $n$ and $0\leq j\leq n$.
Hence, there exists a positive number $\mu$ such that 
\begin{equation}\label{eq3.15}
\Re \benj < -\mu, 
\end{equation}
for all $n\geq 1$ and $0\leq j \leq n$.
\end{lemma}

From the preceding lemma, for $n\geq 1$ we have
\begin{equation}\label{eq3.16}
\begin{aligned}
\mathcal{K}_n(r,s)&=\frac{1}{r}e^{-s(r-1)}
\frac{\prod_{j=1}^n \left(s-\frac{1}{r}\alnj\right)}
{-\prod_{j=0}^n(s-\benj)}\\
&=-\frac{1}{r}e^{-s(r-1)}\frac{1}{s-\beta_{n,0}}
\prod_{j=1}^n 
\left(1+\left(\benj-\frac{1}{r}\alnj\right)\frac{1}{s-\benj}\right).
\end{aligned}
\end{equation}

One could carry out a partial fraction expansion for the right hand side 
of \eqref{eq3.16} to obtain
\begin{equation}\label{eq3.19}
\mathcal{K}_n(r,s)
=\frac{1}{r}e^{-s(r-1)}\sum_{j=0}^n \frac{\bnj(r)}{s-\benj},
\end{equation}
where the coefficients $\bnj$ are given by the formula
\begin{equation}\label{eq3.20}
\begin{aligned}
\bnj(r)&=-\frac{\prod_{k=1}^n \left(\benj-\frac{1}{r}\alnk\right)}
{\prod_{k=0,k\neq j}^n(\benj-\benk)}\\
&=\frac{p_n(\benj r)}{r^n q'_{n+1}(\benj)}\\
&=re^{\benj (r-1)}\frac{k_n(\benj r)}{D'_n(\benj)}.
\end{aligned}
\end{equation}
This would yield
\begin{equation}\label{eq3.21}
\vnm=\frac{1}{r}
\sum_{j=0}^n
\bnj(r) \int_0^{t-r+1}
e^{\benj (t-r+1-\tau)}\gnm(\tau)d\tau, \qquad n>0.
\end{equation}
Unfortunately, the 
coefficients $\bnj$ ($j=0,\cdots,n$) behave as badly as the coefficients
$\anj$ defined in \eqref{eq2.10} for the Dirichlet problem. 
That is, catastrophic cancellation
in \eqref{eq3.21} makes it ill-suited for numerical computation.

Fortunately, 
as in section \ref{sec:recurrence}, we can compute $\vnm$ without catastrophic
cancellation using the following recurrence ($\beta_{0,0}=0$):
\begin{equation}\label{eq3.17}
\begin{aligned}
\psi_0(t)&=\int_0^t e^{\beta_{n,0}(t-\tau)} \gnm(\tau)d\tau,\\
\psi_j(t)&=\psi_{j-1}(t)+\left(\benj-\frac{1}{r}\alnj\right)\int_0^t 
e^{\benj(t-\tau)}\psi_{j-1}(\tau)d\tau, \quad j=1,\cdots,n,
\end{aligned}
\end{equation}
with
\begin{equation}\label{eq3.18}
\vnm(r,t)=-\frac{1}{r}\psi_n(t-r+1).
\end{equation}
We leave the derivation of the recurrence to the reader.
\begin{remark}
It is possible to write down a system of ODEs that is equivalent to \eqref{eq3.17}.
We omit details since the derivation is straightforward and we prefer
the recurrence for numerical purposes in any case.
\end{remark} 
\section{A numerical method} \label{sec:num}

In order to carry out the recurrences 
\eqref{eq3.2.3} or \eqref{eq3.17}, we first need to compute 
to compute the zeros of $k_n(z)$ and $D_n(z)$. 
The following lemma provides asymptotic approximations of the zeros
of these two functions, which we will use as initial guesses followed by 
a simple Newton iteration. In practice, we have found that six Newton steps
are sufficient to achieve double precision accuracy for $n < 10,000$.

\begin{lemma}\label{lem6.1}
\emph{(Asymptotic distribution of the zeros of $k_n(z)$ and $D_n(z)$, 
adapted from \cite{jiang1,tokita}); see also the appendix.}
\begin{enumerate}
\item 
The zeros of 
$k_{n}(z)$ have the following asymptotic expansion
\begin{equation}\label{eq6.1}
\alnj \sim n( z(\zeta_{j}) + O(n^{-1})), \text{
$n\rightarrow\infty$},
\end{equation}
uniformly in $j$, where $\zeta_{j}$ is defined by the formula
\begin{equation}\label{eq6.2}
\zeta_{j} = e^{-2\pi i/3}\left(n+ \frac {1}{2}\right)^{-2/3}a_{j},
\end{equation}
$a_{j}$ is the $j$th negative zero of 
the Airy function $Ai$
whose asymptotic expansion is given by the formula
\begin{equation}\label{eq6.3}
a_{j}\sim-(\frac{3\pi}{2})^{2/3}(j-\frac{1}{4})^{2/3}
+O(j^{-4/3}),
\end{equation}
and $z(\zeta)$ is obtained from inverting the equation
\begin{equation}\label{eq6.4}
\frac{2}{3}\zeta^{3/2}=\ln{\frac{i(1+\sqrt{1+z^2})}{z}} - \sqrt{1+z^2},
\end{equation}
where the branch is chosen so that $\zeta$ is real when $z$ is
positive imaginary. In other words, $z(\zeta)$ lies on
the curve whose parametric equation is
\begin{equation}\label{eq6.5}
z(t)=-(t^2-t\tanh{t})^{1/2}\pm i(t\coth{t}-t^2)^{1/2},
\end{equation}
where $t\in[0,t_{0}]$ and $t_{0}=1.19968\ldots$ is the positive
root of $t=\coth{t}$. 

\item
The zeros of 
$D_n(z)=zk'_{n}(z)+k_n(z)$ have the asymptotic expansion
\begin{equation}\label{eq6.6}
\benj \sim n( z(\xi_{j}) + O(n^{-1})), \text{
$n\rightarrow\infty$},
\end{equation}
uniformly in $j$, where $\xi_{j}$ is defined by the formula
\begin{equation}\label{eq6.7}
\xi_{j} = e^{-2\pi i/3}\left(n+ \frac {1}{2}\right)^{-2/3}b_{j},
\end{equation}
 and $b_{j}$ is the $j$th negative zero of 
the first derivative of the Airy function $Ai'$
whose asymptotic expansion is given by the formula
\begin{equation}\label{eq6.8}
b_{j}\sim-(\frac{3\pi}{2})^{2/3}(j-\frac{3}{4})^{2/3}
+O(j^{-4/3}),
\end{equation}
and $z(\xi)$ is defined as in \eqref{eq6.4} with $\zeta$ 
replaced by $\xi$.
\end{enumerate}
\end{lemma}
Figure \ref{fig6.1} shows the zeros of $k_{10}(z)$, $D_{10}(z)$, $k_{11}(z)$ 
and $D_{11}(z)$.

\begin{figure}[!ht]
\centerline{\includegraphics[trim=120 280 120 260, clip, height=3in]
{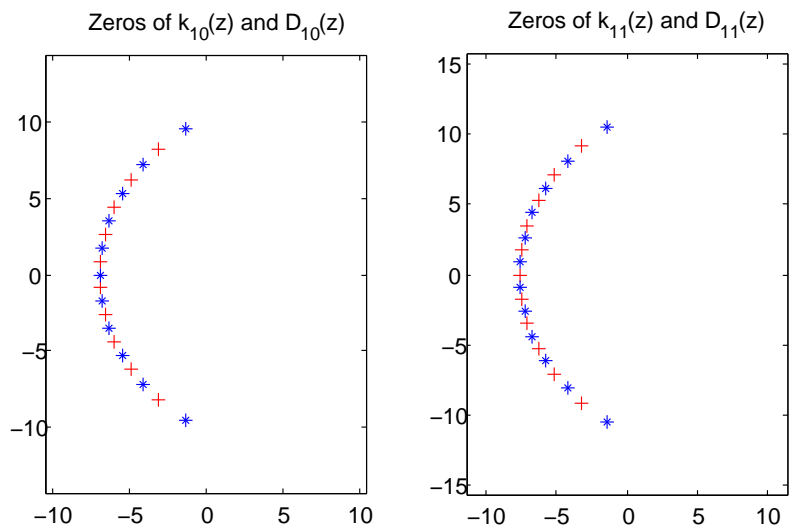}} 
\caption{\emph{Zeros of $k_{10}(z)$, $D_{10}(z)$, $k_{11}(z)$ and $D_{11}(z)$.
The zeros of $k_n$ are marked by red $+$, and the zeros of $D_n$ are marked by 
blue $\ast$.  } } \label{fig6.1}
\end{figure}

\subsection{Marching in time}

We now present a high-order discretization scheme for computing
$\unm$ and $\vnm$. We will only discuss the computation of $\unm$ in detail,
since the treatment of $\vnm$ is analogous.
Recall that the relevant recurrence relations are \eqref{eq3.2.3}
and \eqref{eq3.2.4}. To proceed, we first introduce
the auxillary functions 
\begin{equation}\label{eq7.1}
h_j(t)=\int_0^t e^{\alnj(t-\tau)}\phi_{j-1}(\tau)d\tau, \qquad j=1,\cdots,n.
\end{equation}
Then, \eqref{eq3.2.3} becomes
\begin{equation}\label{eq7.2}
\phi_j(t)=\phi_{j-1}(t)+\left(1-\frac{1}{r}\right)\alnj h_j(t).
\end{equation}
It is easy to check that $h_j(k\Delta t)$ satisfies the recurrence relation
\begin{equation}\label{eq7.3}
h_j(k\Delta t)=e^{\alnj \Delta t} h_j((k-1)\Delta t) + \int_{(k-1)\Delta t}
^{k \Delta t} e^{\alnj (k\Delta t-\tau)} \phi_{j-1}(\tau)d\tau.
\end{equation}
Thus, we need only consider the calculation of the 
integral over $[(k-1)\Delta t, k\Delta t]$. For this,
we interpolate $\phi_{j-1}(\tau)$ by a polynomial of degree $p-1$ 
with the shifted and scaled Legendre nodes as interpolation nodes. That is,
\begin{equation}\label{eq7.4}
\begin{aligned}
\phi_{j-1}(\tau)&\approx \sum_{i=0}^{p-1}c_i P_i\left(
\frac{2}{\Delta t}(\tau-(k-\frac{1}{2})\Delta t)\right)\\
&=\sum_{i=0}^{p-1}\sum_{l=1}^p u_{il}
\phi_{j-1}((k-1)\Delta t+\Delta t(1+x_l)/2) 
P_i\left(\frac{2}{\Delta t}(\tau-(k-\frac{1}{2})\Delta t)\right),
\end{aligned}
\end{equation}
where $x_l$ ($l=1,\cdots,p$) are the standard Legendre nodes
on $[-1,1]$ and $u_{il}$ is the $(i,l)$ entry of the matrix 
converting function values to the coefficients of a
Legendre expansion.

Substituting \eqref{eq7.4} into the integral on the 
right side of \eqref{eq7.3} and simplifying, we obtain
\begin{equation}\label{eq7.5}
\begin{aligned}
&\int_{(k-1)\Delta t}^{k\Delta t} e^{\alpha(k\Delta t-\tau)} 
\phi_{j-1}(\tau)d\tau\\
&\approx \sum_{l=1}^p u_{il} \phi_{j-1}((k-1)\Delta t+\Delta t(1+x_l)/2) \\
&\cdot \frac{\Delta t}{2}\sum_{i=0}^{p-1}  
\int_{-1}^1 e^{\alnj \frac{\Delta t}{2}(1-y)} P_i(y) dy\\
&=\sum_{l=1}^p q_l(\alnj) \phi_{j-1}((k-1)\Delta t+\Delta t(1+x_l)/2),
\end{aligned}
\end{equation}
where the coefficients $q_l$ ($l=1,\cdots,p$) are given
by the formula 
\begin{equation}\label{eq7.6}
q_l(\alnj)=
\frac{\Delta t}{2}\sum_{i=0}^{p-1}u_{il}
 \int_{-1}^1 e^{\alnj \frac{\Delta t}{2}(1-y)} P_i(y) dy.
\end{equation}
Substituting \eqref{eq7.5} into \eqref{eq7.3}, we obtain
\begin{equation}\label{eq7.7}
h_j(k\Delta t)=e^{\alnj \Delta t} h_j((k-1)\Delta t) +
\sum_{l=1}^p q_l(\alnj) \phi_{j-1}((k-1)\Delta t+\Delta t(1+x_l)/2).
\end{equation}
In order to be able to use \eqref{eq7.7}, 
we need to calculate $\phi_{j-1}((k-1)\Delta t+\Delta t/2(1+x_l))$.
For this, we can again apply the recurrence \eqref{eq3.2.3}
and obtain
\begin{equation}\label{eq7.8}
\begin{aligned}
&\phi_{0}((k-1)\Delta t+\Delta t(1+x_l)/2)=f_{nm}((k-1)\Delta t+\Delta t/2(1+x_l)),\\
&\phi_j((k-1)\Delta t+\Delta t(1+x_l)/2)=\phi_{j-1}((k-1)\Delta t+\Delta t/2(1+x_l))\\
&+\left(1-\frac{1}{r}\right)\alnj 
e^{\alnj \Delta t (1+x_l)/2} h_j((k-1)\Delta t)\\
&+\left(1-\frac{1}{r}\right)\alnj \sum_{s=1}^p w_{ls}(\alnj)\phi_{j-1}((k-1)\Delta t+\Delta t/2(1+x_s)),
\end{aligned}
\end{equation}
where the coefficients $w_{ls}$, for $l,s \in \{1,\dots,p\}$, are given
by the formula
\begin{equation}\label{eq7.9}
w_{ls}(\alnj)=
\frac{\Delta t}{2}\sum_{i=0}^{p-1}u_{is}
 \int_{-1}^{x_l} e^{\alnj \frac{\Delta t}{2}(x_l-y)} P_i(y) dy.
\end{equation}

In summary, the algorithm for computing $\unm(r,T)$ proceeds in two
stages: a precomputation stage and a time-marching stage. 

\begin{algorithm}[!ht]

\vspace{.1in}

\caption{Precomputation phase} 
\label{alg1}
\begin{algorithmic}[1]
\REQUIRE For spherical harmonic order $n$, time step $\Delta t$,
and desired order of accuracy $p$, precompute
the coefficients needed in the marching scheme for $\unm$.

\STATE Compute and store the zeros $\alnj$ ($j=1,\cdots,n$) of $k_n$.
\STATE Compute the Legendre nodes $x_l$ ($l=1,\cdots,p$) and 
the $p\times p$ matrix $u$ which converts 
function values to the coefficients of the corresponding
Legendre expansion.
\FOR{$j=1:n$ and $i=1:p$}
\STATE Compute the integrals
 $\int_{-1}^1 e^{\alnj \frac{\Delta t}{2}(1-y)} P_i(y) dy$.
\ENDFOR
\FOR{$j=1:n$ and $i=1:p$}
\STATE Compute and store the coefficients $q_l(\alnj)$ defined in \eqref{eq7.6}.
\ENDFOR
\FOR{$j=1:n$, $i=1:p$, and $l=1:p$}
\STATE Compute the integrals
 $\int_{-1}^{x_l} e^{\alnj \frac{\Delta t}{2}(x_l-y)} P_i(y) dy$.
\ENDFOR
\FOR{$j=1:n$, $s=1:p$, and $l=1:p$}
\STATE Compute and store the coefficients $w_{ls}(\alnj)$ defined in \eqref{eq7.9}.
\ENDFOR
\FOR{$j=1:p$}
\STATE Compute and store $e^{\alnj \Delta t}$.
\ENDFOR
\FOR{$j=1:n$ and $l=1:p$}
\STATE Compute and store $e^{\alnj \Delta t (1+x_l)/2}$.
\ENDFOR
\end{algorithmic}
\end{algorithm}

\begin{algorithm}[!ht]

\vspace{.2in}

\caption{Marching in time} 
\label{alg2}
\begin{algorithmic}[1]
\REQUIRE Given $n$, $r$, $T$, 
the desired order of accuracy $p$, and the number of desired time steps $N_T$,
compute the spherical harmonic mode $\unm$ at $(r,T)$ defined by \eqref{eq3.2.4}.
\STATE Set $\Delta t=(T-r+1)/N_T$.
\FOR{$j=1:n$}
\STATE Set $h_j=0$.
\ENDFOR
\FOR{$k=1:N_T$}
\FOR{$l=1:p$}
\STATE Evaluate the boundary data $\fnm((k-1)\Delta t+\Delta t(1+x_l)/2)$
by computing the spherical harmonic transform of the Dirichlet data $f$,
and set $\phi_{0}((k-1)\Delta t+\Delta t(1+x_l)/2)
=f_{nm}((k-1)\Delta t+\Delta t/2(1+x_l))$.
\ENDFOR
\FOR{$j=1:n$ and $l=1:p$}
\STATE Use \eqref{eq7.8} to compute $\phi_j((k-1)\Delta t+\Delta t(1+x_l)/2)$
\ENDFOR
\FOR{$j=1:n$}
\STATE Use \eqref{eq7.7} to update $h_j$.
\ENDFOR
\ENDFOR
\STATE Set $\unm=\fnm(N_T \Delta t)$.
\FOR{$j=1:n$}
\STATE Compute $\unm=\unm+\left(1-\frac{1}{r}\right)\alnj h_j$.
\ENDFOR
\STATE Compute $\unm=\unm/r$.
\end{algorithmic}
\end{algorithm}

\clearpage

\subsection{Computational complexity}  

For each spherical harmonic mode, it is easy to see that the precomputation
cost is $O(np^2)$ and the marching cost is $O(np^2 N_T)$, where $p$ is the
desired order of accuracy and $N_T$ is the total number of time steps. Thus, if
we truncate the spherical harmonic expansion order at $N$, then 
the precomputation cost is $O(N^2 p^2)$ and the marching cost is 
$O(N^3 p^2 N_T)$.
The cost of computing the spherical harmonic transform of the 
boundary data at all times is $O(N^3 N_T p)$ and the cost of the inverse 
spherical harmonic transform at the final time is $O(N^3)$. Summing all these
factors up, we observe that the total computational cost of our algorithm
is $O(N^3 p^2 N_T)$.

\section{Numerical examples}
We have implemented the above algorithm in Fortran for both the Dirichlet and 
Robin problems governed by the scalar wave equation.
To test its convergence and stability, we consider the exact solution
\begin{equation}
 u(x,t)=\sum_{i=1}^2 c_i e^{-(t-t_i-|x-y_i|)^2/a_i}
\cos(k_i(t-|x-y_i|))/|x-y_i| 
\label{uexact}
\end{equation}
with $y_1=(0.3,-0.5,0.6)$, $t_1=1.2$, $a_1=0.05$, $k_1=100$, and
$y_2=(-0.4,-0.5,0.7)$, $t_2=3.2$, $a_2=0.28$, $k_2=80$.
The numerical solution is computed on a sphere of radius 
$r=100$ at $t=103$. 

Tables \ref{tab8.1}-\ref{tab8.4} show the relative $L^2$ error of the numerical 
solution of the scalar wave equation for varying values of $N$, the order of
the spherical expansion and total number of time steps. 
Note that the solution is oscillatory in both space 
and time, so that finite difference or finite element
methods would have difficulty computing the solution in the far field with high precision
because of numerical dispersion errors.
In Tables \ref{tab8.1} and \ref{tab8.3}, the order of temporal convolution
is $p=10$ and we break the time interval $[99,103]$ into $200$ equispaced subintervals 
(yielding a total of $2000$ discretization points in time).
In Tables \ref{tab8.2} and \ref{tab8.4}, we use $80$ terms in the spherical 
harmonic expansions. These tables show that numerical solution converges 
spectrally fast to the exact solution. 

\begin{table}[ht]
\center\begin{tabular}{|c|c|c|c|c|c|c|}
\hline
$N_S$ & 102400 & 129600 &  160000 & 193600 & 230400 & 270400\\
\hline
$N$ & 80 &     90 &    100 &    110 &    120 &    130\\
\hline
$E$ &  0.84E-01&  0.65E-03&  0.12E-05&  0.64E-09 & 0.89E-12 & 0.88E-12\\
\hline
\end{tabular}
\caption{Relative $L^2$ error of the numerical solution of the Dirichlet problem 
with increasing spherical harmonic expansion order $N$. $N_S$ is the total
number of discretization on the unit sphere. Since the discretization error
is usually greater than the truncation error, $N_\theta=N_\phi$ is chosen
to be $4N$. Thus $N_S=16N^2$. The total number of discretization
points in time is $N_T=2000$.}
\label{tab8.1}
\end{table}

\begin{table}[ht]
\center\begin{tabular}{|c|c|c|c|c|c|c|}
\hline
$N_T$ & 250 &  500 & 750& 1000&1500&2000\\
\hline
$E$ & 0.19E+00& 0.12E-03&  0.15E-05&  0.30E-07& 0.47E-10&0.88E-12\\
\hline
\end{tabular}
\caption{Relative $L^2$ error of the numerical solution of the Dirichlet 
problem as a function of the total number of discretization points in time. 
Here, the spherical harmonic expansion order was set to $125$ and the total
of number of discretization points on the unit sphere is $N_S=250000$. The
order of integration for temporal convolution is fixed at $p=10$.}
\label{tab8.2}
\end{table}

\begin{table}[ht]
\center\begin{tabular}{|c|c|c|c|c|c|c|}
\hline
$N_S$ & 102400& 129600&  160000 &193600 & 230400 &  270400\\
\hline
$N$ & 80 &     90 &    100 &    110 &    120 &    130\\
\hline
$E$ & 0.84E-01 & 0.65E-03 & 0.12E-05 & 0.64E-09 & 0.71E-12& 0.70E-12\\
\hline
\end{tabular}
\caption{Relative $L^2$ error of the numerical solution of the Robin
problem 
with increasing spherical harmonic expansion order $N$. $N_S$ is the total
number of discretization on the unit sphere. Since the discretization error
is usually greater than the truncation error, $N_\theta=N_\phi$ is chosen
to be $4N$. Thus $N_S=16N^2$. The total number of discretization
points in time is $N_T=2000$.}
\label{tab8.3}
\end{table}

\begin{table}[ht]
\center\begin{tabular}{|c|c|c|c|c|c|c|}
\hline
$N_T$ & 250 &  500 & 750& 1000& 1250&1500\\
\hline
$E$ & 0.92E-02&  0.13E-05 & 0.41E-07&  0.15E-08 & 0.58E-10& 0.33E-11\\ 
\hline
\end{tabular}
\caption{Relative $L^2$ error of the numerical solution of the Robin
problem as a function of the total number of discretization points in time. Here, the
spherical harmonic expansion orer is $125$ and the total
of number of discretization points on the unit sphere is $N_S=250000$. The
order of integration for temporal convolution is $p=10$.}
\label{tab8.4}
\end{table}

\begin{figure}[!ht]
\centerline{\includegraphics[height=2in]
{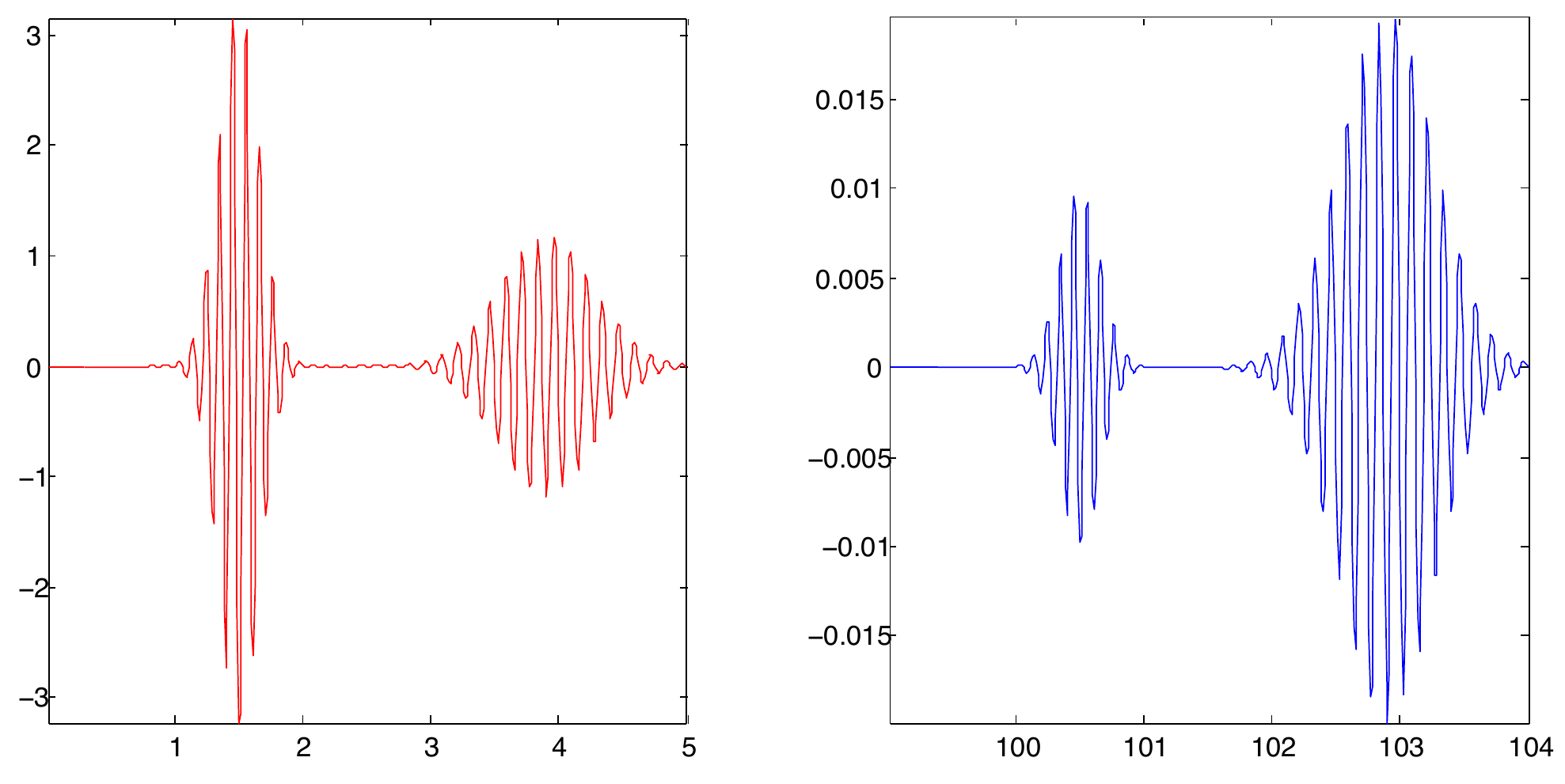}} 
\caption{The left-hand plot shows the value of the boundary data
at the north pole of the 
unit sphere as a function of time, and the right-hand plot shows the 
solution at the north pole of the outer sphere of radius $r=100$.
The exact solution is of the same form as (\ref{uexact}) - that is,
induced by two sources in the interior of the unit sphere.
\label{fig8.2}}
\end{figure}

\begin{figure}[!ht]
\centerline{\includegraphics[height=2in]
{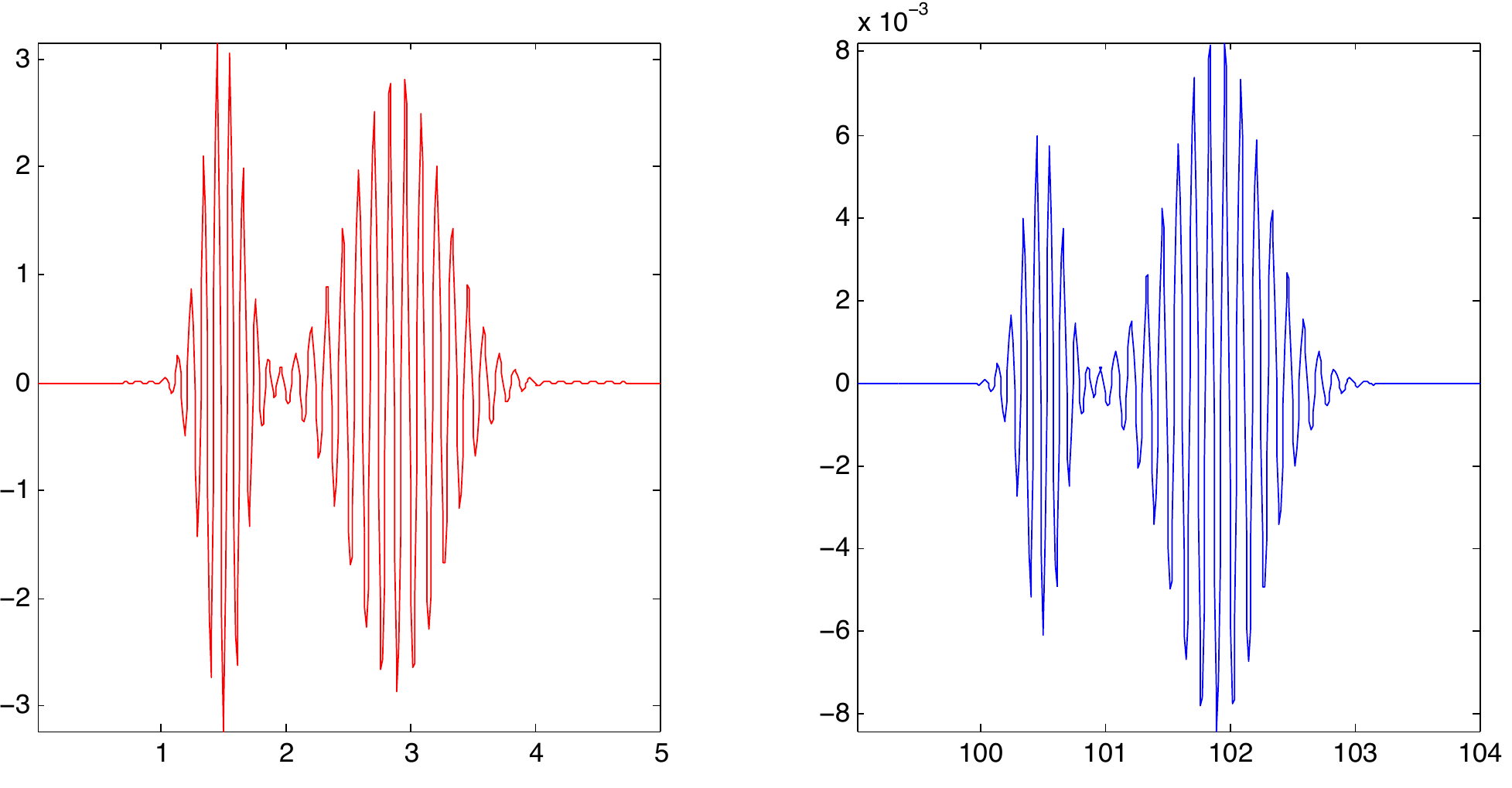}} 
\caption{The analog of Fig. \ref{fig8.2} for a ``true" scattering problem.
Dirichlet boundary conditions are generated by
two exterior sources placed on the $z$-axis, at $(0,0,1.3)$ and 
$(0,0,1.7)$.
The left-hand plot shows the value of the boundary data at the north pole of the 
unit sphere as a function of time, and the right-hand plot shows the 
solution at the north pole of the outer sphere of radius $r=100$.
\label{fig8.3}}
\end{figure}

\begin{figure}[!ht]
\centerline{\includegraphics[height=2in]
{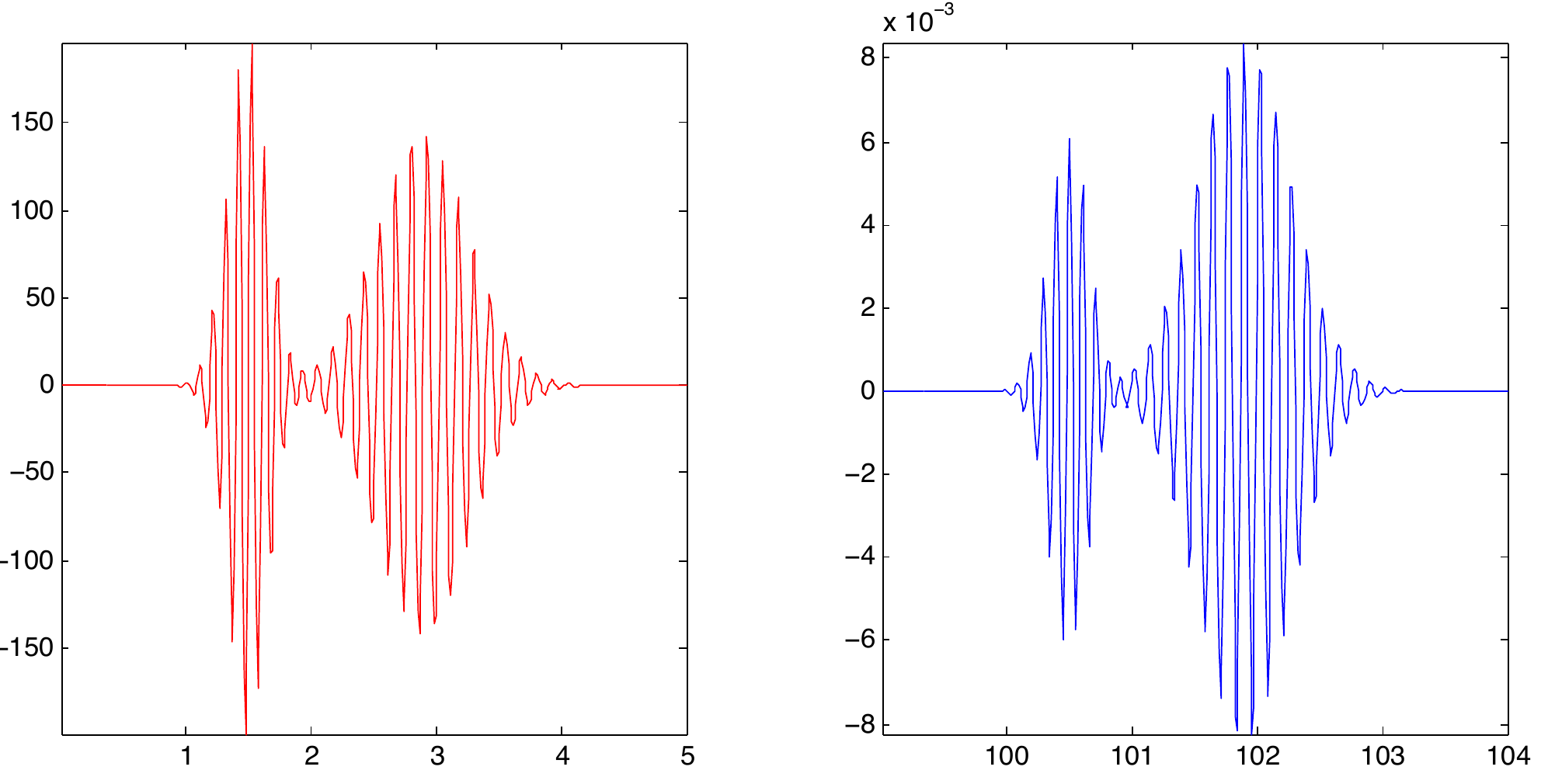}} 
\caption{The analog of Fig. \ref{fig8.2} for a ``true" scattering problem.
Robin boundary conditions are generated by
two exterior sources placed on the $z$-axis, at $(0,0,1.3)$ and 
$(0,0,1.7)$.
The left-hand plot shows the value of the boundary data at the north pole of the 
unit sphere as a function of time, and the right-hand plot shows the 
solution at the north pole of the outer sphere of radius $r=100$.
\label{fig8.4}}
\end{figure}

\begin{figure}[!ht]
\centerline{\includegraphics[height=2in]
{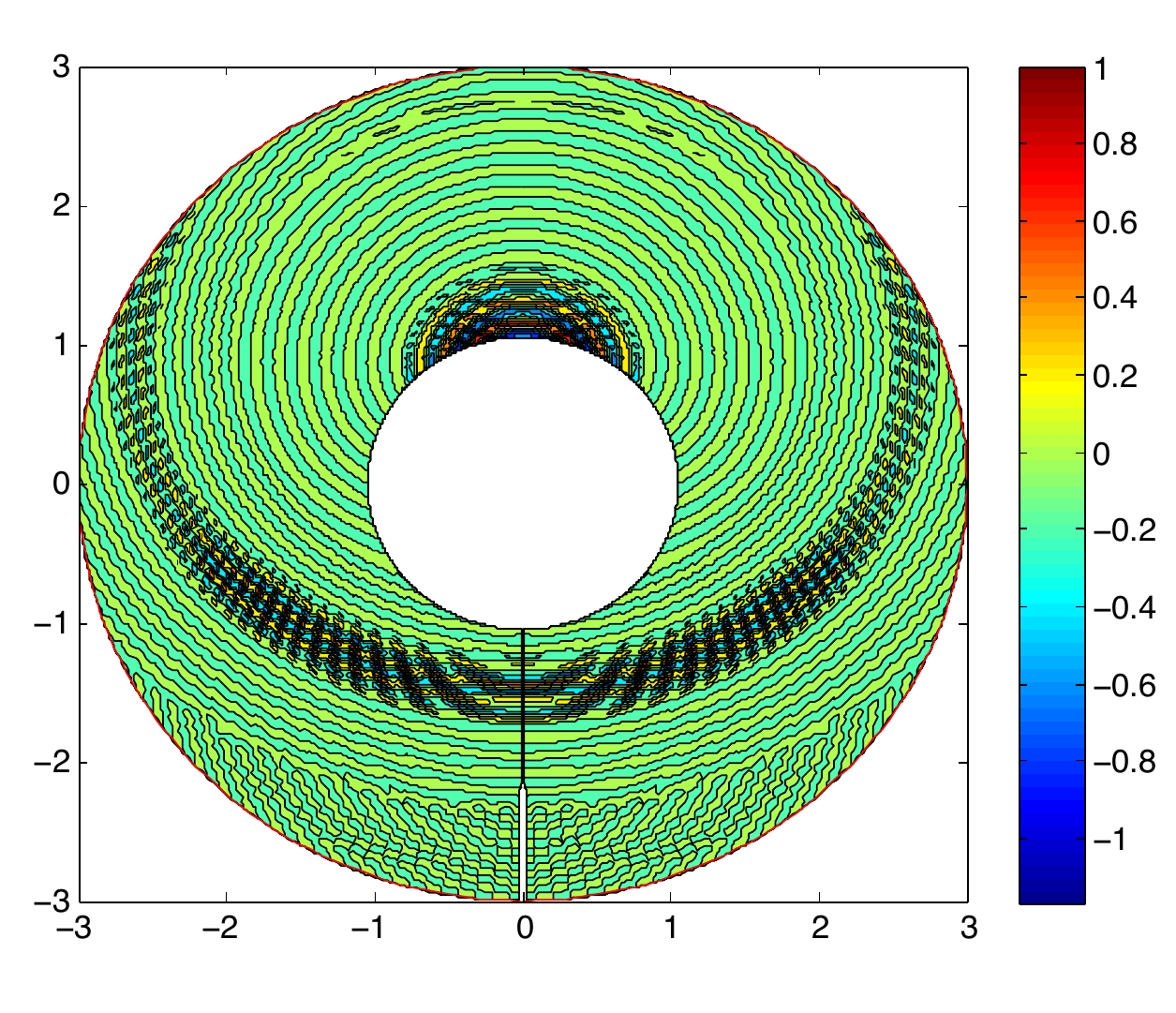}} 
\caption{We plot the solution of the field scattered by the unit sphere
in the $xz$-plane within the annular region $1 < r < 3$ at $t=4$.
with boundary data as in Fig. \ref{fig8.3}.
Note that the domain is approximately 50 wavelengths in size.
\label{fig8.5}}
\end{figure}

\section{Conclusions}
We have presented an analytic solution for the scalar wave equation 
in the exterior of a sphere in a form that is numerically tractable and
permits high order accuracy even for objects many wavelengths in size.
Aside from its intrinsic interest in single or multiple scattering from
a collection of spheres, our algorithm provides a useful reference solution
for any numerical method designed to solve problems of exterior scattering.
At the present time, such codes are typically tested by Fourier transformation
after a long-time simulation and comparison with a set of single frequency
solutions computed by separation of variables applied to the Helmholtz equation.

\begin{remark}
An exception is the work of Sauter and Veit
\cite{SAUTERVEIT}, who make use of a formulation equivalent to that of Wilcox
to develop a benchmark solution for a time-domain integral equation 
solver which can be applied to scattering from general geometries.
Exponential ill-conditioning is avoided by considering only
low-order spherical harmonic expansions. 
Recently, Grote and Sim \cite{GROTESIM} 
have also used an approach based on the local exact radiation boundary conditions
proposed in \cite{HHex} to develop 
a new hybrid asymptotic/finite difference formalism for multiple scattering 
in the time domain. The advantage of the Grote-Sim method is that spherical harmonic transformations
are unnecessary and the evaluation formulas can be localized in angle. However, 
they also restrict their attention to low-order expansions, and our
preliminary experiments using their formulas indicate a loss of conditioning for $n$ large.
(The loss of conditioning presumably also applies to the radiation boundary conditions in \cite{HHex}.)
The method developed here should be of immediate use in both contexts
\end{remark}

As implemented above, our algorithm has $O(N^3N_T)$ complexity. It is possible,
however, to reduce the cost to $O(N^2 \log N N_T)$. This requires
the use of a fast spherical harmonic transform
(see, for example, \cite{tygert1} and references therein). 
With this fast algorithm, the cost
of each spherical harmonic transform is reduced from $O(N^3)$ to
$O(N^2 \log N)$. 
Second, we believe that the convolution kernels can be compressed 
as in \cite{alpert}, so that they involve only $O(\log n)$ 
modes for each $n$ for a given precision. We note that compressions  
for $n=64$ and various radii are reported in \cite{BFL-QG}, both for the
scalar wave equation considered here (which they call the flat-space wave equation)
and for wave equations with Zerilli and Regge-Wheeler potentials. In the
latter cases, compressed kernels are also constructed for smaller values of $n$,
as the exact kernels do not have rational transforms. Tabulated coefficients
required for implementing the compressed kernels 
may be found online \cite{Compressed}.

For the extension of the present approach to the full
Maxwell equations, see \cite{GHJ2}. 
Software implementing the algorithm of the present paper will be made available upon request.

\section{Appendix: asymptotic analysis of exponential growth of the
coefficients $a_{n,j}$ in \eqref{eq2.10}}

An alternative analysis of the instability phenomenon can be carried
out using the uniform asymptotic expansions of the Bessel functons due
to Olver \cite{olver}. We first recall the relationship between $K_{n+1/2}$
and the Hankel function, $H_{n+1/2}^{(1)}$:
\begin{equation}
K_{n+1/2}(z) = \frac {\pi i}{2} e^{i(n+1/2)\pi} H_{n+1/2}^{(1)} (iz) .
\end{equation}
Thus the residues we wish to estimate are given by
\begin{equation}
a_{n,j}(r) = \sqrt{r} e^{(r-1) \alpha_{n,j}} 
\frac {K_{n+1/2}(\alpha_{n,j} r)}{K_{n+1/2}'(\alpha_{n,j})} = -i \sqrt{r} 
e^{(r-1) \alpha_{n,j}}
\frac {H_{n+1/2}^{(1)}(i \alpha_{n,j} r)}{H_{n+1/2}^{(1) \ \prime}(i\alpha_{n,j})} .
\end{equation} 
To approximate these for $n \gg 1$ we use (see \cite{olver}):
\begin{eqnarray}
H_{n+1/2}^{(1)} \left( \left( n+ \frac {1}{2} \right) w \right) \sim & &
\label{asH} \\ 2e^{- \pi i/3} \left( n + \frac {1}{2} \right)^{-1/3} 
\left( \frac {4 \zeta}{1-w^2} \right)^{1/4} {\rm Ai} \left( e^{2 \pi i/3} \left(
n + \frac {1}{2} \right)^{2/3} \zeta \right) , & & \nonumber \\
H_{n+1/2}^{(1) \ \prime} \left( \left( n+ \frac {1}{2} \right) w \right) \sim & &
\label{asdH} \\ \frac {4 e^{- 2 \pi i/3}}{w} 
\left( n + \frac {1}{2} \right)^{-2/3} 
\left( \frac {4 \zeta}{1-w^2} \right)^{-1/4} {\rm Ai} ' 
\left( e^{2 \pi i/3} \left(
n + \frac {1}{2} \right)^{2/3} \zeta \right) , & & \nonumber
\end{eqnarray}
which hold uniformly in $\arrowvert {\rm arg}(w) \arrowvert < \pi - \delta$;
thus in particular they hold in $\Re z < 0$ where we will be using them. 
Here $\zeta$ is given by (\ref{eq6.4}) with the replacement $z=i w$.

To proceed we recall the basic properties of the Airy function, ${\rm Ai}(y)$,
which are listed in the Appendix of \cite{olver} as well as
\cite[Ch. 10]{handbook}:
\begin{description}
\item[i.] ${\rm Ai}(y)$ has infinitely many zeros which lie on the negative real
axis. For $j$ large the jth zero, $a_j$, of ${\rm Ai}(y)$ satisfies (\ref{eq6.3})
and the derivative satisfies
\begin{equation}
{\rm Ai} '(a_j) \sim (-1)^{j-1} \frac {1}{\sqrt{\pi}} \left( \frac {3}{2} \pi 
j \right)^{1/6} .
\end{equation}
\item[ii.] For $\arrowvert {\rm arg}(y) \arrowvert < \pi$ the function 
${\rm Ai}(y)$ satisfies the asymptotic formula
\begin{equation}
{\rm Ai}(y) \sim \frac {1}{2 \sqrt{\pi}} y^{-1/4} e^{- \frac {2}{3} y^{3/2}} , \ \
\arrowvert y \arrowvert \gg 1 . \label{Aiasy}
\end{equation}
\end{description} 

Using (i.) we deduce that the poles, $\alpha_{n,j}$, are asymptotically 
given by (\ref{eq6.1}) and approximately lie on the curve $nz(t)$ where $z(t)$
is defined in (\ref{eq6.5}). This is the curve for which 
$e^{2 \pi i/3} \zeta(z)$ is real and negative.   

To evaluate the residues we must calculate using (\ref{asH}),(\ref{asdH})
\begin{eqnarray}
a_{n,j} & \sim & \sqrt{r} e^{(r-1) \alpha_{n,j}} \alpha_{n,j} \left( n + \frac {1}{2}
\right)^{1/3} \left( \frac {\zeta \ \zeta_r}{ (1+ \alpha_{n,j}^2 r^2 )
(1+ \alpha_{n,j}^2 )} \right)^{1/4} \nonumber \\ & & \times
\frac {{\rm Ai} \left( e^{2 \pi i/3} \left(
n+ \frac {1}{2} \right)^{2/3} \zeta_r \right)}{{\rm Ai}'(a_j)} ,
\end{eqnarray}
where we have introduced
\begin{equation}
\zeta_r = \zeta( \alpha_{n,j} r ) .
\end{equation}
Obviously the scaling $z \rightarrow rz$ moves $\zeta_r$ off the curve
where the argument of the Airy function is real and negative. Thus using
(\ref{Aiasy}) and (\ref{eq6.4}) we deduce that the asymptotic formula for
$a_{n,j}$ contains an exponential term
\begin{eqnarray}
a_{n,j} & \propto & e^{(r-1) \alpha_{n,j}- \frac {2}{3} (n+1/2) \zeta_r^{3/2}} 
\label{resexp} \\
& = & \exp\left[ \left( n+ \frac {1}{2} \right) \eta_{n,j} (r) \right]
\nonumber
\end{eqnarray}
where
\begin{equation}
\eta_{n,j}(r)=\left( (r-1) 
\tilde{\alpha}_{n,j} -  
\ln\left( \frac {1 + \sqrt{ 1 + \tilde{\alpha}_{n,j}^2 r^2 }}
{i \tilde{\alpha}_{n,j} r} \right) + \sqrt{ 1 + \tilde{\alpha}_{n,j}^2 r^2} 
\right) . \label{etadef}
\end{equation}
Here we have introduced $\tilde{\alpha}_{n,j}=\left(n + \frac {1}{2} \right)^{-1}
\alpha_{n,j}$. 

Finally, we consider the real part of the expression in
parentheses on the second line of (\ref{resexp}). 
In particular we 
replace $\tilde{\alpha}_{n,j}$ by a continuous variable $\alpha$ traversing
the scaled curve, $z(t)$, containing the approximate zeros. Then the function
$\eta$ depends only on $r$ and the coordinate describing the curve; in
particular it is independent of $n$ and $j$. In Fig. \ref{resasy2}
we plot the real part of $\eta$ scaled by $\log_{10}{e}$ for $r=2$. 
This can be compared with Fig. \ref{fig2.2} by scaling both axes by $n=80$ and
recognizing the vertical axis as the base ten logarithm. We then observe  
good agreement with the numerical results. The maximum value plotted in
Figure \ref{resasy2} is approximately $.13$, which is the predicted slope
of the straight line plotted in Fig. \ref{fig2.1}. Again the agreement is good. We
note that increasing $r$ makes the problem somewhat worse; the scaled
maximum real part is approximately $.23$ for $r=5$ and $.27$ for $r=20$. 

\begin{figure}[!ht]
\centerline{\includegraphics[height=2in]
{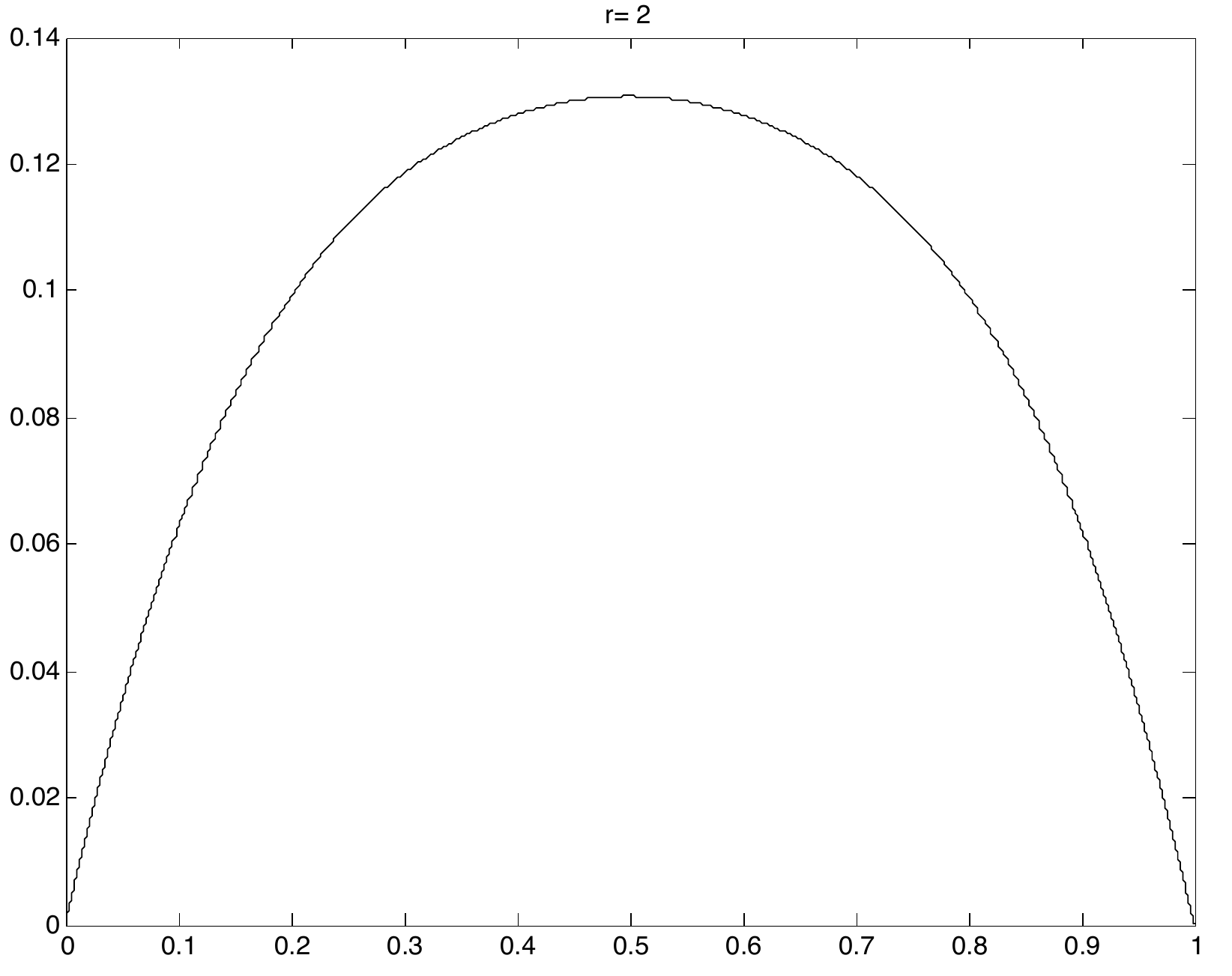}} 
\caption{\emph{Base 10 exponent of the exponential part of the asymptotic
formula for the residue scaled by $n+1/2$.}}\label{resasy2}
\end{figure}

\end{document}